\documentclass{amsart}
\usepackage[latin1]{inputenc}
\usepackage[english]{babel}
\usepackage{amstext}
\usepackage{amsfonts}
\usepackage{amssymb}
\usepackage{amsxtra}

\def\RR{\M{R}}

\def\ZZ{\M{Z}}

\newtheorem{theorem}{Theorem}[section]
\newtheorem{lemma}[theorem]{Lemma}
\newtheorem{corollary}[theorem]{Corollary}
\newtheorem{proposition}[theorem]{Proposition}
\newtheorem{example}[theorem]{Example}
\newtheorem{remark}[theorem]{Remark}

\def\bit{\begin{itemize}}
\def\eit{\end{itemize}}
\reversemarginpar   
\def\bc{\begin{center}}
\def\ec{\end{center}}
\def\bthm{\begin{theorem}}
\def\ethm{\end{theorem}}
\def\bcor{\begin{corollary}}
\def\ecor{\end{corollary}}
\def\bprop{\begin{proposition}}
\def\eprop{\end{proposition}}
\def\blem{\begin{lemma}}
\def\elem{\end{lemma}}

\def\brem{\begin{remark}}
\def\erem{\end{remark}}
\def\bdes{\begin{description}}
\def\edes{\end{description}}

\def\iti{\item[(i)]}
\def\itii{\item[(ii)]}

\def\beq{\begin{equation}}
\def\eeq{\end{equation}}

\def\ben{\begin{enumerate}}
\def\een{\end{enumerate}}

\def\beqar{\begin{eqnarray}}
\def\eeqar{\end{eqnarray}}
\def\beqarr{\begin{eqnarray*}}
\def\eeqarr{\end{eqnarray*}}

\def\RR{{\mathbb R}}  

  \def\cF{\mathcal{F}}

 \def\cN{\mathcal{N}}

\def\kL{\mathcal{L}}

\def\P{{\mathsf P}} 
\def\Q{{\mathsf Q}} 

\def\E{{\mathsf E}} 

\def\ZZ{{\mathbb Z}}       

\def\p{\varphi}

\def\part{\partial}

\def\d#1dt{\frac{d#1}{dt}}    

\title{Excited Brownian Motions}

\author{Olivier RAIMOND}
\address{Laboratoire Modal'X, Universit\'e Paris Ouest Nanterre La D\'efense, B\^atiment G, 200 avenue de la R\'epublique 92000 Nanterre, France.}
\email{olivier.raimond@u-paris10.fr}
\author{Bruno SCHAPIRA}
\address{D\'epartement de Math\'ematiques, B\^at. 425, Universit\'e Paris-Sud 11, F-91405 Orsay, cedex, France. }
\email{bruno.schapira@math.u-psud.fr}

\begin{document}

\begin{abstract}
We study a natural continuous time version of excited random walks, introduced by Norris, Rogers and Williams about twenty years ago. We obtain a
necessary and sufficient condition for recurrence and for positive
speed. Condition under which a central limit theorem holds is also given.
 These results are analogous to the ones obtained for excited (or cookie) random walks.
\end{abstract}



\maketitle

\section{Introduction}
Random processes that interact with their past trajectory have been
studied a lot these past years. Reinforced random walks were
introduced by Coppersmith and Diaconis, and then studied by
Pemantle, Davis and many other authors (see the recent survey
\cite{Pem}). Some examples of time and space-continuous processes
defined by a stochastic differential equation have also been
studied. For example the self-interacting (or attracting)
diffusions studied by Bena\"im, Ledoux and Raimond (see
\cite{BLR,BR2,BR3}) and also by Cranston, Le Jan, Herrmann, Kurtzmann and
Roynette (see \cite{CLJ,HR,K}), are processes defined by a stochastic
differential equation for which the drift term is a function of the
present position and of the occupation measure of the past process. Another example which is not solution of a stochastic differential equation was studied by T\'oth and Werner \cite{TW}. This process is a continuous version of some self-interacting random walks studied by T\'oth (see for instance the survey \cite{T}).

Carmona, Petit and Yor \cite{CPY}, Davis \cite{D2}, and Perman and Werner \cite{PW} (see also other references therein) studied what they called a
perturbed Brownian motion, which is the real valued process $X$
defined by
$$X_t=B_t+\alpha\max_{s\leq t} X_s + \beta \min_{s\leq t}X_s,$$
where $B$ is a Brownian motion. This process can be viewed has a weak
limit of once edge-reinforced random walks on $\ZZ$ (see in particular \cite{D1,D2,W}).

More recently, excited (or cookie) random walks were introduced by
Benjamini and Wilson \cite{BW}, and then further studied first on
$\ZZ$ \cite{BaS1,BaS2,D,KM,KZer,MPiVa,Zer1}, but also in higher dimension and on trees
(see in particular \cite{ABK, BerRa, BaS3,Ko1,Ko2,V,Zer2}). In this class of walks, the transition probabilities
depend on the number of times the walk has visited the present site.
In particular Zerner \cite{Zer1} and later Kosygina and Zerner \cite{KZer} showed that on $\ZZ$,
if $p_i$ is the probability to go from $x$ to $x+1$ after the
$i$-th visit to $x$ and if either $p_i\ge 1/2$ for all $i$ or $p_i=1/2$ for $i$ large enough, then the walk is a.s. recurrent if, and only
if, $$\delta:=\sum_i (2p_i-1)\in [-1,1],$$ and it is a.s. transient
otherwise. Next Basdevant and Singh \cite{BaS1} and Kosygina and Zerner \cite{KZer} proved that the random walk has positive speed if and only if
$$\sum_i (2p_i-1)\notin [-2,2]$$
and that it satisfies a central limit theorem (see \cite{KZer}) if
$$\sum_i (2p_i-1)\notin [-4,4].$$
Recently Dolgopyat \cite{D} proved a functional central limit theorem in the recurrent case (more precisely when $\delta<1$ and $p_i\ge 1/2$ for all $i$), identifying the limiting process as a perturbed Brownian motion (with $0\le \alpha=-\beta<1$).

We study what could be another continuous time version of cookie random walks. Excited
Brownian motions considered here are defined by the stochastic
differential equation
$$dX_t=dB_t+\p(X_t,L_t^{X_t})\ dt,$$
for some bounded and measurable $\p$, where $B$ is a Brownian motion
and $L_t^x$ is the local time in $x$ at time $t$ of $X$. These processes were introduced by Norris, Rogers and Williams \cite{NRW}, in connection with the volume excluded problem (see \cite{NRW0}). In fact they considered the case when $\p$ is only supposed to be locally bounded. Then under the assumption that $\p$ only depends on the second coordinate, that it is continuous, nonnegative and that its integral is strictly larger than $1$, they proved that the process  $X$ is well defined, transient and when $\p$ is nondecreasing they proved a law of large numbers. Along the proof they obtained also a Ray--Knight type theorem. Later Hu and Yor proved a central limit theorem \cite{HY}, assuming also that $\p$ is nonnegative, continuous and nondecreasing. Here we mainly concentrate on the case when $\p$ is bounded, but we also discuss the case when $\p$ is possibly unbounded at the end of the paper. Our main result (see Theorem \ref{theorec} and \ref{LLN} below) is the
\begin{theorem}
Assume that $\p$ is constant in the first variable, Borel-measurable and bounded. Set
$$C_1^\pm:=\int_0^\infty \exp \left[ \mp\int_0^x \frac{dl}{l}\int_0^l \p(0,u)\ du \right] \ dx.$$
Then
\begin{enumerate}
\item the process $X$ is almost surely recurrent if, and only if, $C_1^+=C_1^-=+\infty$,
\item we have $\lim_{t\to+\infty}X_t = +\infty$ (resp. $-\infty$) almost surely if, and only if, $C_1^+<+\infty$ (resp. $C_1^-<\infty$),
\item we have $\lim_{t\to +\infty} X_t/t=v >0$ (resp. $<0$) almost surely if, and only if, $C_2^+<\infty$ (resp. $C_2^-<\infty$), where
$$C_2^\pm:=\int_0^\infty x\exp \left[ \mp\int_0^x \frac{dl}{l}\int_0^l \p(0,u)\ du\right] \ dx,$$
and in this case $v=C_1^+/C_2^+$ (resp. $v=-C_1^-/C_2^-$).
\end{enumerate}
\end{theorem}
An immediate consequence is the following
\bcor
\label{corintro}
Assume that $\p$ is constant in the first variable, Borel-measurable and bounded,
and moreover that $\p$ is nonnegative or compactly supported. Then
\begin{enumerate}
\item recurrence of $X$ is equivalent to $$\int_0^\infty \p(0,l)\ d l\in [-1,1],$$
\item and nonzero speed is equivalent to
$$\int_0^\infty \p(0,l)\ d l\notin [-2,2].$$
\end{enumerate}
\ecor

In section \ref{tlc}, we give a sufficient condition ensuring that $X$ satisfies a central limit theorem (see Theorem \ref{theotcl}). In particular (see Proposition \ref{proptcl}) it implies the
\bcor 
Under the hypotheses of Corollary \ref{corintro}
\begin{itemize}
\item[(3)] the process $X$ satisfies a central limit theorem if
$$\int_0^\infty \p(0,l)\ d l\notin [-4,4].$$ 
\end{itemize}
\ecor

Note that in these results we assume that $\p$ is constant in the first variable. We will also consider the case of non-homogeneous $\p$ (i.e. non necessarily constant in the first variable) in Section 5, for which we obtain some partial results. In particular if we make the additional hypothesis that $\p$ is nonnegative, then we give a sufficient condition for recurrence (see Corollary \ref{cor}).  

\vspace{0.2cm}
The paper is organized as follows. In the next section we define excited Brownian motions and give some elementary properties.
In section 3 we describe the law of the excursions of these processes above or below some level. In section 4 we study the property of recurrence and prove a general 0-1 law. In section 5 we study the particular case of nonnegative $\p$ and obtain a necessary and sufficient criterion for recurrence or transience. The tools used there are elementary martingale techniques, very similar to the techniques in \cite{Zer1}, and partly apply for non-homogeneous $\p$, as we mentioned above. However they do not apply
well for general $\p$. In section 6 we give a criterion for recurrence and transience without any condition on the sign of $\p$, but assuming that it is constant in the first variable. For this we use the tools developed in sections 3 and 4 and arguments similar to those in \cite{NRW}, in particular a Ray--Knight theorem. We also obtain a law of large numbers with an explicit expression for the speed. Here again, the proof follows the arguments given in \cite{NRW}, but we partly extend and simplify them. In section 7 we give a central limit theorem, following Hu and Yor \cite{HY}. 

\vspace{0.1cm}
\noindent \textit{Acknowledgments: We would like to thank Itai Benjamini for his suggestion to consider the
processes studied in this paper. We thank also Martin Zerner and Arvind Singh for bringing to our attention the papers \cite{NRW} and \cite{HY}.}

\section{Definitions and first properties}
\label{secdef}
Denote by $(\Omega,\cF,\Q_x)$ the Wiener space, where $\Q_x$ is the law of a real Brownian motion started at $x$.
Define $X_t(\omega)=\omega(t)$ for all $t\ge 0$ and $(\cF_t,t\ge 0)$ the filtration associated to $X$.
In the following, $L_t^y$ denotes the local time process of $X$ at level $y$ and at time $t$.

Let $\Lambda$ be the set of measurable bounded functions
$\p:\RR\times\RR^+\to \RR$.  The subset of $\Lambda$ of nonnegative functions will be denoted by $\Lambda^+$. We will denote by
$\Lambda_c$ and $\Lambda^+_c$ the sets of functions $\p$ in
$\Lambda$ (resp. in $\Lambda^+$) such that $\p(x,l)$ is a constant
function of $x$.

For $\p \in \Lambda$, set
$$M^\p_t:=\exp\left(\int_0^t \p(X_s,L_s^{X_s})\ dX_s - \frac{1}{2} \int_0^t\p^2(X_s,L_s^{X_s})\ ds\right).$$
Then $(M^\p_t,t\ge 0)$ is an $(\cF_t,t\ge 0)$ martingale of mean one w.r.t. $Q_x$ (see Corollary (1.16) p.333 in \cite{RY}).
So we can define $\P^\p_{x,t}$ as the probability measure on $(\Omega,\cF_t)$ having a density $M^\p_t$ with respect to $\Q_x$ restricted to $\cF_t$.
By consistency, it is possible to construct a (unique) probability measure $\P^\p_x$ on $\Omega$, such that $\P^\p_x$ restricted to $\cF_t$ is $\P^\p_{x,t}$.
By the transformation of drift formula (Girsanov Theorem), one proves that under  $\P^\p_x$,
$$B_t=X_t-x-\int_0^t \p(X_s,L_s^{X_s})\ ds,$$
is a Brownian motion started at $0$.

This proves (the uniqueness follows by a similar argument: start
with $\P^\p_x$ the law of a solution and then construct
$\P^\p_{x,t}$) the following \bprop Let $(x,\p) \in
\RR\times\Lambda$. Then there is a unique solution $(X,B)$ to the
equation
\begin{equation}
\label{EDS}
X_t=x+B_t+\int_0^t \p(X_s,L_s^{X_s})\ ds,
\end{equation}
with $L_t^y$ the local time of $X$ at level $y$ and at time $t$, and such that $B$ is a Brownian motion started at $0$.
\eprop

For clarity, we will sometimes write $\Q$ for $\Q_0$ and $\P$ for
$\P_0^\p$, if there is no ambiguity on $\p$. We will use the
notation $\E$ and $\E_x^\p$ for the expectations with respect to
$\P$ and $\P_x^\p$. For other probability measures $\mu$, the
expectation of a random variable $Z$ will simply be denoted by $\mu(Z)$.

\medskip
Let $(x,\p)\in\RR\times\Lambda$. Under $\P^\p_x$, for all stopping times $T$, on the event
$\{T<\infty\}$,
\begin{eqnarray}
\label{Markov}
\hbox{the law of } (X_{t+T}, t\geq 0)
\hbox{ given } \cF_T \hbox{ is } \P_{X_T}^{\p_T},
\end{eqnarray}
where $\p_T\in \Lambda$ is defined by
$$\p_T(y,l)=\p(y,L_T^y+l).$$
Note that $(X_t,\p_t)$ is a Markov process and that (\ref{Markov})
is just the strong Markov property for this process.

\medskip
We denote by $D_t$ the drift accumulated at time $t$ :
$$D_t=\int_0^t \p(X_s,L_s^{X_s})\ ds.$$

\blem \label{ocformula} Set $h(x,l)=\int_0^l \p(x,u)\ du$. The drift term $D_t$ is also equal to
$$\int_\RR h(x,L_t^x)\ dx.$$
\elem

\begin{proof} This follows from the occupation times formula given in exercise
(1.15) in the chapter VI of Revuz--Yor \cite{RY}. \end{proof}

\medskip
In the following, we set for any Borel set $A$ of $\RR$,
$$D_t^A=\int_A h(x,L_t^x)\ dx.$$
We will use also the notation $D^+_t$, $D^-_t$ and $D_t^k$, $k\in \ZZ$, respectively for $D_t^{\RR^+}$, $D_t^{\RR^-}$ and $D_t^{(k,k+1)}$.
Note that (this is still a consequence of Exercise (1.15) Chapter VI in \cite{RY})
$$D_t^A=\int_0^t \p(X_s,L_s^{X_s}) 1_A(X_s)\ ds.$$

\begin{lemma}
\label{convergence}
Let $(\p_n)_{n\ge 0} \in \Lambda$ be a sequence of functions. Assume that for a.e. $x$ and $l$, $\p_n(x,l)$ converges toward $\p(x,l)$ when $n\to +\infty$. Assume also that
$$\sup_{x,n,u} |\p_n(x,u)| <+\infty.$$
Then $\P_x^{\p_n}$ converges weakly on $\Omega$ toward $\P_x^\p$ for all $x$.
\end{lemma}
\begin{proof} Let $Z$ be a bounded continuous $\cF_t$-measurable random variable. We have to prove that $\E_x^{\p_n}(Z)$ converges towards $\E_x^{\p}(Z)$.
Since $\E_x^{\p_n}(Z)$ and $\E_x^{\p}(Z)$ are respectively equal to
$\Q_x(ZM^{\p_n}_t)$ and to $\Q_x(ZM^{\p}_t)$ it suffices to prove
that $M^{\p_n}_t$ converges in $L^2(\Q_x)$ towards $M^\p_t$. By using It\^o
calculus, we get:
\begin{eqnarray*}
\Q_x[(M^{\p_n}_t-M^{\p}_t)^2]&=&\Q_x\left[\int_0^t d<M^{\p_n}-M^\p>_s\right]\\
                             &=& \Q_x\left[\int_0^t \left(M^{\p_n}_s\p_n(X_s,L_s^{X_s})-M^\p_s\p(X_s,L_s^{X_s})\right)^2\ ds\right].
\end{eqnarray*}
By using next the basic inequality $(a+b)^2\le 2(a^2+b^2)$ we get
$$\frac{d}{dt} \Q_x[(M^{\p_n}_t-M^{\p}_t)^2] \leq 2\Q_x[(\p_n-\p)^2(X_t,L_t^{X_t})(M^\p_t)^2]+C\Q_x[(M^{\p_n}_t-M^{\p}_t)^2],$$
for some constant $C>0$. By dominated convergence, the first term of the right hand side converges to $0$. We conclude by using Gronwall's lemma.
\end{proof}

\section{Construction with excursions}
Define the processes $A^+$ and $A^-$ as follows:
$$A^+_t=\int_0^t 1_{\{X_s>0\}}\ ds\quad \textrm{and} \quad A^-_t=\int_0^t1_{\{X_s<0\}}\  ds.$$
Define the right-continuous inverses of $A^+$ and $A^-$ as
\beq
\label{kappa}
\kappa^+(t)=\inf\{u>0 \mid  A_u^+>t\} \quad \textrm{and} \quad \kappa^-(t)=\inf\{u>0 \mid  A_u^->t\}.
\eeq
Define the two processes $X^+$ and $X^-$ by
$$X^+_t=X_{\kappa^+(t)} \quad \textrm{and} \quad X^-_t=X_{\kappa^-(t)}.$$
Denote by $\Q^+$ and $\Q^-$ the laws respectively of $X^+$ and $X^-$
under $\Q$, and let $\Q^+_t$ and $\Q^-_t$ respectively be their
restrictions to $\cF_t$ (then $\Q^\pm_t$ is the law of
$(X^\pm_s;s\leq t)$). It is known (see section 8.5 in \cite{Y}, with
$\mu=1$) that $\Q^+$ (resp. $\Q^-$) is the law of a Brownian motion
reflected above $0$ (resp. below $0$) and started at $0$. The
process $\beta$, defined by
$$\beta_t:=X_t - L^0_t \qquad \hbox{(resp. } \beta_t:=X_t + L^0_t\hbox{)},$$
(recall that $L^0$ is the local time process in $0$ of $X$) is a
Brownian motion under $\Q^+$ (resp. under $\Q^-$). Denote by
$N_t^\p$ the martingale on $(\Omega,\cF,\Q^{\pm})$ defined by
$$N_t^\p := \exp\left(\int_0^t \p(X_s,L_s^{X_s})\ d\beta_s - \frac{1}{2} \int_0^t\p^2(X_s,L_s^{X_s})\ ds\right).$$
Let $\P^{\p,\pm}$ be the measures whose restrictions $\P^{\p,\pm}_t$ to $\cF_t$ are defined by
$$\P^{\p,\pm}_t := N^\p_t\cdot \Q^{\pm}_t \qquad \text{for all } t\ge 0.$$
Note that, by using Girsanov Theorem, on the space $(\Omega,\cF,\P^{\p,\pm})$,
\begin{equation} \label{edspm}
X_t=\beta_t^{\pm} \pm L^0_t + \int_0^t \p(X_s,L_s^{X_s}) ds, \end{equation}
with $\beta_t^{\pm}$ a Brownian motion.

Set $\widetilde{\P}^\p:=\P^{\p,+}\otimes\P^{\p,-}$ and let $(X^1,X^2)$ be the canonical process of law $\widetilde{\P}^\p$. Then $X^1$ and $X^2$ are independent and respectively distributed like $\P^{\p,+}$ and $\P^{\p,-}$. Denote by $L^{(1)}$ and $L^{(2)}$ the local time processes in $0$ of $X^1$ and $X^2$, and define their right continuous inverses $\tau_1$ and $\tau_2$ by
$$\tau_i(s)=\inf\{t \mid L^{(i)}_t>s\} \quad\hbox{ for all }s\ge 0\textrm{ and } i\in\{1,2\}.$$
Note that this definition extends to random times $s$. Denote by $e^1$ and $e^2$ their excursion processes out of $0$: for $s\le L^{(i)}_\infty$,
$$ e^i_s(u)=X^i_{\tau_i(s-)+u} \quad \hbox{ for all } u\in (0,\tau_i(s)-\tau_i(s-))\hbox{ and }i\in\{1,2\},$$
if $\tau_i(s)-\tau_i(s-)>0$, and $e^i_s=0$ otherwise. Let now $e$ be the excursion process obtained by adding $e^1$ and $e^2$: for all $s\ge 0$, $e_s:=e_s^1+e_s^2$ (note that with probability $1$, for all $s\ge 0$, $e_s^1$ or $e_s^2$ equals zero).

Denote by $\Xi$ the measurable transformation that reconstructs a
process out of its excursion process (see \cite{RY} Proposition
$(2.5)$ p.482). Note that $\Xi$ is not a one to one map, it is only
surjective (think of processes having an infinite excursion out of
$0$, in which case $L^0_\infty<\infty$). In particular, if $\tilde{e}^1$ and $\tilde{e}^2$ are defined by
\begin{eqnarray*}
\tilde{e}^1_s &=& e^1_s~ 1_{\{s\le L^{(1)}_\infty\wedge L^{(2)}_\infty\}}\\
\hbox{and } \quad  \tilde{e}^2_s &=& e^2_s~ 1_{\{s\le L^{(1)}_\infty\wedge L^{(2)}_\infty\}},
\end{eqnarray*}
then $\Xi(\tilde{e}^1+\tilde{e}^2)=\Xi(e)$. Recall now that $\P=\P_0^\p$ and that $(\Xi e)^+_t$ and $(\Xi e)^-_t$ have been defined at the beginning of this section.

\begin{proposition}
\label{excursion}
Denote by $L$ the local time in $0$ of $\Xi e$.
Then the following holds
\begin{itemize}
\item[(i)] The law of the process $\Xi e$ is $\P$.
\item[(ii)]  For all $t\in\RR^+\cup\{\infty\}$,
$L_t=L^{(1)}_t\wedge L^{(2)}_t$.
\item[(iii)] $\left((\Xi e)^+_t,t \le\tau_1(L_\infty) \right)=\left(X^1_t,t \le \tau_1(L_\infty)\right)$.
\item[(iv)] $\left((\Xi e)^-_t,t \le \tau_2(L_\infty) \right)=\left(X^2_t,t \le \tau_2(L_\infty)\right)$.
\end{itemize}
\end{proposition}
\noindent Note that $t\le \tau_1(L_\infty)$ (resp. $t\le \tau_2(L_\infty)$) is equivalent to say that $L^{(1)}_t\le L^{(2)}_\infty$ (resp. $L^{(2)}_t\le L^{(1)}_\infty$).
\begin{proof}
Assume first that $\p(x,l)=0$ if $x\in (-c,c)$, for some constant $c>0$.
For any $\epsilon\in (-c,c)$, define a process $X^\epsilon$ as follows:
set $T_0^\epsilon=0$ and for $n\ge 1$,
$$S_n^\epsilon=\inf\{t\ge T_{n-1}^\epsilon \mid X_t\in\{-\epsilon,\epsilon\}\},$$
$$T_n^\epsilon = \inf\{t\ge S_n^\epsilon \mid X_t=0\}.$$
Define also
$$A_\epsilon(t)=\sum_{n\ge 1} (T_n^\epsilon \wedge t - S_n^\epsilon\wedge t),$$
and let $\kappa_\epsilon(t)$ be the right-continuous inverse of $A_\epsilon$. Then set
$$X^{\epsilon}_t:=X_{\kappa_\epsilon(t)}.$$
Now observe that during each time-interval $(S_n^\epsilon,
T_n^\epsilon)$, the local time in $0$ of $X$ cannot increase, and that, 
for $t\in (A_\epsilon(S_n^\epsilon),A_\epsilon(T_n^\epsilon))$,
$\kappa_\epsilon(t)=S_n^\epsilon+(t-A_\epsilon(S_n^\epsilon))$. So by
\eqref{edspm}, during the intervals
$(A_\epsilon(S_n^\epsilon),A_\epsilon(T_n^\epsilon))$, if $X$
follows the law of $\Xi e$, then $X^{\epsilon}$ is solution of the
SDE
$$d X^{\epsilon}_t = d R^\epsilon_t + \p\left(X^{\epsilon}_t,L_{\kappa_\epsilon(t)}^{X^{\epsilon}_t}\right) \ d t,$$
where $L_\cdot^\cdot$ is the local time of $X$, and $R^\epsilon$ is
the Brownian motion defined by
$$R^\epsilon_t=\sum_{k=1}^{n-1}(B_{T_k^\epsilon}-B_{S_k^\epsilon}) +
(B_{\kappa_\epsilon(t)}-B_{S_n^\epsilon}),$$ for $t\in
(A_\epsilon(S_n^\epsilon),A_\epsilon(T_n^\epsilon))$.

Denote by $L^{\epsilon,\cdot}_\cdot$ the local time process of
$X^{\epsilon}$. Then
$$L^x_{\kappa_\epsilon(t)}=L^{\epsilon,x}_t \quad \hbox{for all } t\ge 0
\quad \hbox{and all } x\not\in (-c,c).$$ Since $\p(x,l)=0$ when $x\in
(-c,c)$,
$$\p(x,L^x_{\kappa_\epsilon(t)})=\p(x,L^{\epsilon,x}_t) \quad \hbox{for all } t\ge 0
\quad \hbox{and all } x\in \RR.$$

Thus $X^{\epsilon}$ satisfies in fact the SDE: \begin{equation} \label{SDEXeps}
d X^{\epsilon}_t = d R^\epsilon_t +
\p\left(X^{\epsilon}_t,L_t^{\epsilon,X^{\epsilon}_t}\right) \ d t,
\end{equation} up to the first time it hits $0$. Moreover when it hits $0$, it
jumps instantaneously to $\epsilon$ or to $-\epsilon$ with
probability $1/2$, independently of its past trajectory. Note that
this determines the law of $X^\epsilon$.

But it follows also from \eqref{EDS}, that if $X$ has law $\P$, then
$X^\epsilon$ solves as well the SDE \eqref{SDEXeps} up to the first
time it hits $0$, and then jump to $\epsilon$ or $-\epsilon$ with
probability $1/2$ (since there is no drift in $(-c,c)$). Since
moreover $X^\epsilon$ converges to $X$, when $\epsilon$ goes to $0$,
we conclude that the law of $\Xi e$ is $\P$.

To finish the proof of (i),
for $c>0$, define $\p_c$ by
$$\p_c(x,l)=\p(x,l) 1( x\notin [-c,c]).$$
By Lemma \ref{convergence} we know that $\P_0^{\p_c}$ converges toward $\P_0^\p$, when $c\to 0$.
It can be also seen that $\P^{\p_c,+}$ and $\P^{\p_c,-}$ converge respectively towards $\P^{\p,+}$ and $\P^{\p,-}$. Since $\Xi e$ is a measurable transformation of $(X^1,X^2)$, we can conclude.

Assertions (ii), (iii) and (iv) are immediate: in the construction
of $\Xi e$, one needs only to know $e_s$ for $s\le
L_\infty=L^{(1)}_\infty\wedge L^{(2)}_\infty$. So $(\Xi e)^+$ and
$(\Xi e)^-$ can be respectively reconstructed with the positive and
negative excursions of $(e_s,s\le L^{(1)}_\infty\wedge
L^{(2)}_\infty)$.
\end{proof}

\section{A $0-1$ law for recurrence and transience}

Let $\p\in\Lambda$. Consider the events
$R_a:=\{L_\infty^a=+\infty\}$, $a\in \RR$. Using conditional
Borel-Cantelli lemma, one can prove that for all $x,y,z$, and all
$\p\in\Lambda$, $\P_x^\p$-a.s., $R_y = R_z$. In the following, we
will denote by $R$ the event of recurrence ($=R_a$ for all $a$).

We will first study the question of recurrence and transience for
the processes $X^1$ and $X^2$ separately, where $X^1$ and $X^2$ are
independent respectively of law $\P^{\p,+}$ and $\P^{\p,-}$. Note
that we still have for all $x\geq 0$ (resp. $x\leq 0$), and all
$\p\in\Lambda$, $\P^{\p,+}$-a.s. (resp. $\P^{\p,-}$-a.s.), $R_x =
R_0(=R)$. So in all cases, $R$ is the event $\{L^0_\infty=\infty\}$.

Fix $x\in\RR$ and denote by $\p_x$ the function in $\Lambda$ such
that for all $(y,l)\in\RR\times\RR^+$,
$$\p_x(y,l)=\p(x+y,l).$$
In the following, $\P_x^{\p,\pm}$ will denote $\P^{\p_x,\pm}$.
We also set 
$$T_a:=\inf\{t>0\mid X_t=a\} \quad \hbox{for all } a\in \RR.$$
\bprop \label{rec12} For all $x\in\RR$ and all $\p\in\Lambda$, the
following holds 
\begin{equation}
 \P_x^\p(R) = \P_x^{\p,+}(R)\times\P_x^{\p,-}(R). \end{equation} \eprop
\begin{proof} This is a straightforward application of Proposition
\ref{excursion}. Recall that $\tilde{\P}^{\p_x}$ has been defined in the previous section and that $L_\infty$ refers to the process $\Xi(e)$. Now since 
\beqarr
\P_x^\p(R)&=&\tilde{\P}^{\p_x}(L_\infty=\infty)\\
&=&\tilde{\P}^{\p_x}(L^{(1)}_\infty=L^{(2)}_\infty=\infty), 
\eeqarr
and we conclude since $L^{(1)}_\infty$ and $L^{(2)}_\infty$ are
independent.
\end{proof}

For $t>0$ and $a\in\RR$, set $\sigma^{a}_t=\inf\{s>0\mid \int_0^s
1_{\{X_u\ge a\}}\ du >t\}$. Then we have

\blem \label{a+}Let $a\in\RR$, $x,y \le a$ and $\p,\psi\in \Lambda$
be such that $\p(z,l)=\psi(z,l)$ for all $z\geq a$ and all $l\geq
0$. Then $(X_{\sigma_t^a}-a,t\ge 0)$ has the same distribution under
$\P_x^{\p,+}$ and under $\P_y^{\psi,+}$. In particular,
\begin{eqnarray*}
\P_x^{\p,+}(L_\infty^a=\infty)=\P_y^{\psi,+}(L^a_\infty=\infty).
\end{eqnarray*}
\elem
\begin{proof} Just observe that an analogue of Proposition \ref{excursion} can be proved (exactly in the same way) for the 
set of excursions above or below any level $a\in \RR$. In particular this shows that the law of the excursions of $X$ above level $a$ only depends on $\p(z,l)$ or $\psi(z,l)$ for $z\ge a$. The first assertion follows. The second assertion is immediate since the local time in $a$ is a measurable function of $(X_{\sigma_t^a}-a,t\ge 0)$.
\end{proof}

Note that a similar proposition holds for $\P_\cdot^{\p,-}$. A
direct consequence of this lemma is that \bprop \label{xpm} For all
$x\in\RR$, $\P_x^{\p,\pm}(R)=\P^{\p,\pm}(R)$. \eprop

With this in hands, we can prove the

\bprop[Zero-one law] \label{01law} Let $\p\in\Lambda$. Then
$\P^{\p,+}(R)$ and $\P^{\p,-}(R)$ both belong to $\{0,1\}$. \eprop
\begin{proof}
Assume that $\P^{\p,+}(R)>0$. By martingale convergence theorem,
$\P^{\p,+}$-a.s. \beqarr
1_R&=&\lim_{z\to+\infty}\P^{\p,+}(R\mid \cF_{T_z})\\
&=&\lim_{z\to+\infty}\P_z^{\p_{T_z},+}(R)\\
&=&\P^{\p,+}(R), \eeqarr
by application of Lemma \ref{a+} and Proposition
\ref{xpm}. Thus $\P^{\p,+}(R)=1$, which proves the proposition.
\end{proof}

Denote by $T$ the event of transience. Then
$T=R^c=\{L^0_\infty<\infty\}$. Note that $\P^{\p,\pm}$-a.s.,
$T=\{X\to\pm\infty\}$, and that $\P_x^\p$-a.s.,
$T=\{X\to+\infty\}\cup\{X\to -\infty\}$ and that
$\tilde{\P}^\p$-a.s., $\{\Xi e\hbox{ is
transient}\}=\{X^1\to\infty\}\cup\{X^2\to -\infty\}$. Then we have
the criterion

\bprop Let $\p \in \Lambda$. \bit \iti $\P_x^\p(R)=1$ for all $x$ if,
and only if, $\P^{\p,+}(R)=\P^{\p,-}(R)=1$. \itii $\P_x^\p(R)=0$ for
all $x$ if, and only if, $\P^{\p,+}(R)=0$ or $\P^{\p,-}(R)=0$. \eit
\eprop
\begin{proof} This is a straightforward consequence of Propositions \ref{rec12} and \ref{xpm}. Note moreover that Proposition 
\ref{01law} shows that cases $(i)$ and $(ii)$ cover all possibilities.
\end{proof}

Let us remark that when $X$ of law $\P=\P_x^\varphi$ is transient, 
it is possible to have $\P(X\to +\infty)=1-\P(X\to -\infty)\in (0,1)$. For
instance take $\p\in \Lambda$ such that $\P^{\p,+}(R)=0$, i.e. one
has  $\P^{\p,+}$-a.s. $X\to +\infty$. Then define $\psi\in\Lambda$
by $\psi(x,u)=\p(x,u)$ if $x\ge 0$ and $\psi(x,u)=-\p(x,u)$ if $x <
0$. Then by symmetry we have
$$\P_0^\psi(X\to +\infty)=\P_0^\psi(X\to -\infty)=1/2.$$
More generally,
$$\P_0^\p(X\to+\infty) = \tilde{\P}^\p(L^{(1)}_\infty<L^{(2)}_\infty),$$
and
$$\P_0^\p(X\to-\infty) = \tilde{\P}^\p(L^{(1)}_\infty>L^{(2)}_\infty).$$
So, if $L_\infty^{(1)}$ and $L_\infty^{(2)}$ are finite random variables (i.e. if $X^1$ and $X^2$ are transient), since they are independent, these two
probabilities are positive (see Lemma \ref{lemma8} below).

\section{Case when $\p \ge 0$}
In this section, we will consider functions belonging to
$\Lambda^+$. In this case, it is obvious that $\P^{\p,-}(R)=1$. Thus
the recurrence property only depends on
$\P^{\p,+}(R)$. We recall also that $D_t$, $D_t^+$, $D_t^-$ and $D_t^k$ have been defined in Section \ref{secdef} (in particular we stress out that $D_t^k$ has not to be thought as the $k$-th power of $D_t$).  

\blem\label{lem51} Let $\p\in\Lambda^+$, for which there exists
$x_0$ such that for all $x\le x_0$ and all $l\ge 0$,
$\p(x,l)=\p(x_0,l)$, and where $\p(x_0,l)$ is a nonzero function of
$l$. Then, for all $a\geq x$,
$$\E^\p_x(D_{T_a})=a-x.$$ \elem
\begin{proof} Without losing generality, we prove this result for $x=0$. Take $a>0$. We have for all $n$,
\beq
\label{eq1}
0=\E_0^\p(B_{T_a \wedge n})=\E_0^\p(X_{T_a \wedge n})-\E_0^\p(D_{T_a \wedge n}).
\eeq
By monotone convergence, 
\beq 
\label{eq2} 
\lim_{n\to \infty} \E_0^\p(D_{T_a \wedge
n}) = \E_0^\p(D_{T_a}).
\eeq
Thus, if we can prove that
$\lim_{n\to\infty}\E_0^\p(X_{T_a \wedge n})=a$, the lemma will
follow. Since $\P_0^\p$-a.s., $X_{T_a}=a$ and $T_a<\infty$, this is
equivalent to proving that
$$\lim_{n\to\infty}\E_0^\p\left[X_n1_{\{n<T_a\}}\right]=0.$$
For all $n$,
$$\min_{t<T_a} X_t \leq X_n 1_{\{n<T_a\}}\leq a \qquad \P_0^\p-a.s.$$
If one proves that $\min_{t<T_a} X_t$ is integrable, we can conclude
by dominated convergence. One has for all $c>0$, 
\beqarr
\E_0^\p\left[-\min_{t<T_a} X_t\right] &=& \E_0^\p\left[(-\min_{t<T_a} X_t)1_{\{-\min_{t<T_a} X_t\le D_{T_a}/c\}}\right]\\ 
 & +& \E_0^\p\left[(-\min_{t<T_a} X_t)1_{\{-\min_{t<T_a} X_t\ge D_{T_a}/c\}}\right]\\
&\leq& \E_0^\p\left[D_{T_a}/c\right] + \sum_{i\geq 1} i\times\P_0^\p\left[-\min_{t<T_a} X_t \in
[i-1,i),D_{T_a}<ci\right]. \eeqarr 
Now \eqref{eq1} and \eqref{eq2} show that $\E_0^\p[D_{T_a}/c]\leq
a/c<+\infty$. Moreover for all $i\ge 1$,
\begin{equation}
\label{drift} \P_0^\p\left[-\min_{t<T_a} X_t \in [i-1,i),\
D_{T_a}<ci\right]\le \P_0^\p\left[T_{-i+1}<T_a,\ D_{T_a}<ci\right].
\end{equation}
Now on the event $\{T_{-i+1}<T_a\}$,
$$D_{T_a}\ge \sum_{k=1}^{i-1}\left(\int_{-k}^{-k+1} h(y,L_{T_{-k}}^y)\ dy\right) 1_{\{T_{-k}<\infty\}}.$$
Set
$$\alpha_k:= \left(1\wedge \left(\int_{-k}^{-k+1} h(y,L_{T_{-k}}^y)\ dy\right)\right) 1_{\{T_{-k}<\infty\}}.$$

We will prove that there exist positive constants $\alpha$ and $C$,
such that for all integer $k$ greater than $|x_0|$,
\begin{eqnarray}
\label{alpha} \E_0^\p[\alpha_{k+1}\mid \cF_{T_{-k}}] \ge
\alpha 1_{\{T_{-k}<+\infty\}}.
\end{eqnarray}
By \eqref{Markov} we have \beqarr \E_0^\p[\alpha_{k+1}\mid
\cF_{T_{-k}}] &=&
\E_{-k}^{\p_{T_{-k}}}\left[\alpha_{k+1}\right]1_{\{T_{-k}<\infty\}}\\
&=& \E_{0}^{\psi_{k}}\left[\alpha_{1}\right]
1_{\{T_{-k}<\infty\}}, \eeqarr with $\psi_k(y,l)=\p(x_0,l)$ for
$y<0$ and $\psi_k(y,l)=\p(y-k,L^{y-k}_{T_{-k}}+l)$ for $y\ge 0$. Note
that by Proposition \ref{excursion}, $\Xi e$ will reach level $-1$
in finite time if, and only if,
$$L_\infty^{(1)}\ge L^{(2)}_{T_{-1}},$$
where $T_{-1}$ denotes also the hitting time of $-1$ for $X^2$. Thus
$$\E_0^{\psi_{k}}\left[\alpha_{1}\right]
= \widetilde{\E}^{\psi_{k}}\left[F\left(X^2_t, t\leq
T_{-1}\right)1_{\{L_\infty^{(1)}\ge L^{(2)}_{T_{-1}}\}}\right],$$
with
$$F\left(X^2_t, t\leq T_{-1}\right)=1\wedge\int_{-1}^{0} h\left(y,L_{T_{-1}}^{(2),y}\right)\ dy,$$
where $L^{(2),y}$ is the local time in $y$ of $X^2$. It is possible
to couple $X^1$ with a Brownian motion $Y^+$ reflected at $0$,
started at $0$, with drift $\|\p\|_\infty$, and such that $X^1_t\le
Y^+_t$ for all $t\ge 0$. Note that the law of $Y^+$ is
$\P^{\psi,+}$, where $\psi(x,l)=\|\p\|_\infty$ if $x\ge 0$ and
$\psi(x,l)=\p(x_0,l)$ if $x<0$. Moreover,
$$L^{(1)}_{\infty} \ge L^+_\infty,$$
with $L^+$ the local time in $0$ of $Y^+$. These properties imply that
$$ \E_0^{\psi_{k}}\left[\alpha_{1}\right]
\ge \alpha := \E_{0}^{\psi}\left[\alpha_{1}\right].$$ It
remains to see that $\alpha$ is positive. For any $t>0$, it is
larger than
$$\E_0^{\psi}\left[(\alpha_{1}) 1_{\{T_{-1}\le t\}}\right].$$
By using Girsanov's transform it suffices to prove that for some
positive $t$,
$$\Q\left[\alpha_1 1_{\{T_{-1}\le t\}}\right]>0.$$
If not, $\Q\left[\alpha_{1}1_{\{T_{-1}\le t\}}\right]=0$ for all $t$.
Since $\Q[T_{-1}\le t]>0$ for all $t>0$, this implies
$$\int_{-1}^0\Q\left[h(y,L^y_{T_{-1}})1_{\{T_{-1}\le t\}}\right]\ dy=0.$$
By continuity of $h$ and $L$ it suffices to prove that for some $y\in (-1,0)$ and $t>0$,
$$\Q\left[h(y,L^y_{T_{-1}})1_{\{T_{-1}\le t\}}\right]>0.$$
But this is clear for $y=1/2$ for instance, since for any $M>0$, $\Q[L^{1/2}_{T_{-1}}>M]>0$. This proves \eqref{alpha}.

We deduce now from \eqref{alpha} that the process $(M_i,i\ge 0)$ defined by
$$M_i := \sum_{k=1}^{i} (\alpha_k-\alpha 1_{\{T_{-k}<+\infty\}}),$$
is a sub-martingale with respect to the filtration $(\cF_{T_{-i}})_{i\geq 1}$. Moreover on the event $\{T_{-i}<T_a\}$, one has
$$M_i = \sum_{k=1}^{i} (\alpha_k-\alpha).$$
So by taking $c=\alpha/2$ in \eqref{drift} we get
$$\P\left[T_{-i}<T_a,\ D_{T_a}<c(i+1)\right]\le \P[M_i\le -\alpha i /2],$$
for all $i\ge 0$. We conclude now that $\E[-\min_{t<T_a} X_t]$ is finite by using concentration results for martingales with bounded increments (see Theorem 3.10 p.221 in \cite{McD}).
\end{proof}

\brem \emph{As noticed also by Zerner in \cite{Zer1}, the condition
$\p(x_0,\cdot)\neq 0$ is necessary (the result being obviously false
if $\p=0$).} \erem

\brem \emph{This lemma can be extended to the case when $\p$ is
random and stationary in the sense that for all $z$, $\p^z$ and $\p$
have the same law, where $\p^z$ is defined by
$\p^z(x,l)=\p(x+z,l)$.} \erem

\blem \label{lemma6}
Let $\p \in \Lambda^+_c$.
Then
$$\E_0^\p(D_\infty^k) \le 1,$$
for all $k\ge 0$.
\elem
\begin{proof} It is the same proof than for Lemma 11 in Zerner \cite{Zer1}.
We reproduce it here for completeness. Note first that $\E^\p_0[D^k_\infty]=\P^{\p,+}_k(D^0_\infty)$, which does not depend on $k$ since $\p\in\Lambda_c$.
For $K\geq 1$ and $i\leq K-1$,
$$D_{T_K}\geq D^+_{T_K} = \sum_{j=0}^{K-1}D^j_{T_K} \geq \sum_{j=0}^{K-1-i}D^j_{T_K}  \geq \sum_{j=0}^{K-1-i}D^j_{T_{j+i}}.$$
By using Lemma \ref{lem51},
$$K=\E_0^\p[D_{T_K}] \geq  \sum_{j=0}^{K-1-i}\E_0^\p[D^j_{T_{j+i}}]\geq (K-i)\E_0^\p[D^0_{T_i}].$$
Dividing by $K$, letting $K\to \infty$ and then $i\to\infty$, we conclude.
\end{proof}

\blem \label{lemma7} Let $\p\in\Lambda^+$ be such that
$\P^{\p,+}(R)=0$. Then
$$\lim_{z\to\infty}\E_0^\p(D_{T_z}^+)/z=1.$$
If moreover $\p \in \Lambda^+_c$, then
$$\E_0^\p(D_\infty^0)=1.$$
\elem
\begin{proof} The proof of the first part follows the proof of Lemma 6 in Zerner \cite{Zer1}.
We write it here for completeness.
First, since $\E_0^\p(D^+_{T_z})=\P^{\p,+}(D_{T_z})$ and
$\P^{\p,+}(R)$ is a function of $(\p(x,\cdot))_{x\geq 0}$,
we can assume that $\p(x,l)=1$ for all $x\le 0$ and all $l\ge 0$. Thus Lemma \ref{lem51} can be
applied: $\E_0^\p(D_{T_z})=z$. So, it suffices to prove that
$\lim_{z\to\infty}\E_0^\p(D^-_{T_z})/z=0$.
For $i \ge 1$, let $\sigma_i=\inf\{j \ge T_i \mid X_j = 0\}$. We have, for
$z$ an integer, $$ \E_0^\p[D^-_{T_z}] = \sum_{i=0}^{z-1}
\E_0^\p[D^-_{T_{i+1}}-D^-_{T_i}].$$ Note that \beqarr
\E_0^\p[D^-_{T_{i+1}}-D^-_{T_i}] &=&
\E_0^\p\left[1_{\{\sigma_i<T_{i+1}\}}(D^-_{T_{i+1}}-D^-_{T_i})\right]\\
&=& \E_0^\p\left[1_{\{\sigma_i<T_{i+1}\}}\E_0^{\p_{\sigma_i}}(D^-_{T_{i+1}})\right]\\
&\le& (i+1)\P_0^\p(\sigma_i<T_{i+1}), \eeqarr by using again Lemma
\ref{lem51}. Thus it remains to prove that
$$\lim_{z\to\infty}\frac{1}{z}\sum_{i=1}^{z-1}i\P_0^\p(\sigma_i<T_{i+1})=0.$$
Let $Y_i = \P_0^\p[\sigma_i<T_{i+1} \mid \cF_{T_i}]$. Since
$\P_0^\p(R)=0$, the conditional Borel-Cantelli lemma implies that
$\P_0^\p$-a.s., $\sum_iY_i<+\infty$. Since $Y_i\le 1/i$ ($X$ being
greater than a Brownian motion), for all positive $\epsilon$,
$$i\times \P_0^\p(\sigma_i<T_{i+1})\le \epsilon +
\P_0^\p\left(Y_i\ge\epsilon/i\right).$$ This implies that
$$\frac{1}{z}\sum_{i=1}^{z-1}i\P_0^\p(\sigma_i<T_{i+1})
\le \epsilon +\frac{1}{z}\E_0^\p\left[\sum_{i=1}^{z-1}1_{\{Y_i\ge
\frac{\epsilon}{i}\}}\right].$$ But since $\sum_i Y_i<\infty$ a.s., the
density of the $i\le z$ such that $Y_i\ge \epsilon/i$ tends to $0$
when $z$ tends to $\infty$. Thus the preceding sum converges to 0 by dominated convergence.
This concludes the proof of the first part.

The second part is immediate (see Zerner \cite{Zer1} Theorem 12).
Since $\p \in \Lambda_c^+$, for all $K\ge 0$,
$$\E_0^\p[D_\infty^0]=\frac{1}{K} \sum_{k=0}^{K-1} \E_0^\p[D_\infty^k]\ge \frac{1}{K}\E_0^\p\left[D^+_{T_K}\right].$$
We conclude by using the first part of the lemma.
\end{proof}

The next result gives a sufficient condition for recurrence, when we only know that $\p \in \Lambda^+$. For $\p \in \Lambda^+_c$, we will obtain a necessary and sufficient condition in Theorem \ref{theo} below.
\bcor \label{cor} Let $\p\in \Lambda^+$. For $x\in \RR$, set $\delta^x(\p)=\int_{0}^\infty \p(x,u)\ du$. If
$$\liminf_{z\to +\infty} \frac{1}{z}\int_0^z \delta^x(\p)\ dx<1,$$
then $\P^{\p,+}(R)=1$. \ecor
\begin{proof} Since $\P$-a.s. $D_{T_z}^+\leq \int_0^z \delta^x(\p)\ dx$, if $\liminf \frac{1}{z}\int_0^z \delta^x(\p)\ dx<1$, then
$$\liminf \E_0^\p(D^+_{T_z})/z<1.$$
We conclude by using Lemma \ref{lemma7}.
\end{proof}

\blem \label{lem:-1} Let $\p\in\Lambda$ be such that
$\P^{\p,+}(R)=0$. Then
$$\P_0^\p(T_{-1}=+\infty)>0.$$
\elem
\begin{proof} By using Proposition \ref{excursion},
$$\P_0^\p(T_{-1}=+\infty)
= \widetilde{\P}^\p(L^{(1)}_\infty<L^{(2)}_{T_{-1}}),$$ where
$T_{-1}$ denotes also the hitting time of $-1$ for $X^2$. Since
$X^1$ and $X^2$ are independent and since $\P^{\p,+}(R)=0$ implies
that $\widetilde{\P}^\p$-a.s., $L^{(1)}_\infty < +\infty$, it suffices
to prove that for any $l>0$,
$$\P^{\p,-}(L^0_{T_{-1}}>l)>0.$$
Equivalently it suffices to prove that for any $l>0$, there exists $t>0$ such that
$$\P^{\p,-}(L^0_{T_{-1}}>l \hbox{ and } T_{-1}\le t)>0.$$
By absolute continuity of $\P^{\p,-}_{|\cF_t}$ and $\Q^{-}_{|\cF_t}$
this is equivalent to
$$\Q^-(L_{T_{-1}}>l \hbox{ and } T_{-1}\le t)>0.$$
But this is clear, since $L_s$ and $T_{-1}$ have positive densities with respect to Lebesgue measure, for any $s>0$ (see \cite{RY}). This concludes the proof of the lemma.
\end{proof}

\blem \label{lemma8} Let $\p \in \Lambda$ be such that
$\P_0^\p(R)=0$. Then for any $M>0$,
$$\P_0^\p[L_\infty^0 <M] >0.$$
\elem

\begin{proof}
Since $\P_0^\p(R)=0$, Proposition \ref{rec12} shows that $\P^{\p,+}(R)=0$ or $\P^{p,-}(R)=0$. Assume for instance that $\P^{\p,+}(R)=0$, the other case being similar. We have
\beqarr
\P_0^\p[L_\infty^0 <M]
&\geq& \P_0^\p[T_1<+\infty,\ L_{T_1}^0 <M \hbox{ and } X_t>0 \quad \forall t>T_1]\\
&\geq& \E_0^\p\left[1_{\{T_1<+\infty,\ L_{T_1}^0 <M\}} \P_1^{\p_{T_1}}(T_0=\infty)\right].
\eeqarr
But Lemma \ref{lem:-1} implies that a.s., $\P_1^{\p_{T_1}}(T_0=\infty)>0$.
So it remains to prove that $\P_0^\p(T_1<+\infty,\ L_{T_1}^0 <M)>0$.
Like in the previous lemma, by absolute continuity, it suffices to prove that
$$\Q(L_{T_1}^0 <M)>0.$$
But as in the previous lemma, this claim is clear. This concludes the proof of the lemma.
\end{proof}

\noindent Finally we obtain the

\bthm
\label{theo}
Let $\p \in \Lambda^+_c$. Then
$$\P^{\p,+}(R)=1 \quad \Longleftrightarrow \quad \int_0^\infty \p(0,u) \ du \le 1.$$
 \ethm

\begin{proof} We prove this for $\p$ such that $\p(x,l)=\p(0,l)$ for all $x$.
By Lemma \ref{lemma6}, $\E_0^\p(D_\infty^0)\le 1$. But if
$\P^{\p,+}(R)=1$, then $\P^\p_0$-a.s. $L_\infty^x=+\infty$, for
all $x$. So, by using the occupation time formula (see Lemma
\ref{ocformula}) we have
 $\E_0^\p(D_\infty^0)=\int_0^\infty \p(0,u)\ du$. This gives the
necessary condition. Reciprocally, if $\P^{\p,+}(R)=0$, we saw in
Lemma \ref{lemma7} that $\E_0^\p(D_\infty^0)=1$. But by Lemma
\ref{lemma8}, we have
$\E_0^\p(D_\infty^0)=\E_0^\p(h(0,L_\infty^0))<\int_0^\infty \p(0,u)\ du$,
which gives the sufficient condition and concludes the proof of the
theorem.
\end{proof}

Note that if $\p\in\Lambda$ is such that for some $a\in\RR$,
$\p(x,l)=\p(a,l)\ge 0$ for all $x\ge a$ and all $l\ge 0$, then
$$\P^{\p,+}(R)=1 \quad \Longleftrightarrow \quad \int_0^\infty \p(a,u) \ du \le 1.$$
This can be proved using the fact that $\P^{\p,+}(R)=\P^{\p_a,+}(R)$
which does not depend on $\p(x,\cdot)$, for $x<a$.


\brem \emph{With the technique used in this section one could prove as well
that for any $\p \in \Lambda_c$, such that $\int_0^{+\infty} |\p(0,u)|\ du <+\infty$, recurrence
implies $\int_0^\infty \p(0,u) \ du \in [-1,1]$. But as we will see, this last condition is not sufficient for recurrence. In fact the necessary and sufficient condition we will obtain in the next section requires more sophisticated tools. In particular we will use a Ray--Knight theorem.}
\erem


\section{General criterion for recurrence and law of large numbers}

\subsection{A Ray--Knight theorem}
By following the proof of the usual Ray--Knight theorem for Brownian motion given for instance in \cite{RY} Theorem (2.2) p.455, a Ray--Knight theorem can be obtained. Such a theorem is proved in \cite{NRW} (Theorem 2).
For $a\ge 0$ and $0\le x\le a$, let
$$Z^{(a)}_x:=L_{T_a}^{(1),a-x},$$
where we recall that $L^{(1),\cdot}_\cdot$ is the local time process
of $X^1$ and $T_a$ is the first time $X^1$ hits $a$. Recall also
that the law of $X^1$ is $\P^{\p,+}$ and that the function $h$ has been defined in Lemma \ref{ocformula}.
\bthm[Ray--Knight]
\label{rayknight0} The process $(Z^{(a)}_x,x\in [0,a])$ is a
non-homogeneous Markov process started at $0$ solution of the
following SDE:
\begin{eqnarray}
\label{EDSRK0}
dZ^{(a)}_x = 2 \sqrt{Z^{(a)}_x}\ d\beta_x + 2(1-h(a-x,Z^{(a)}_x))\ dx,
\end{eqnarray}
with $\beta$ a Brownian motion. \ethm
\begin{proof} We just sketch the proof. It is actually proved in [NRW2]. Applying Tanaka's formula at time $T_a$:
$$(X_{T_a}-(a-x))^+=\int_0^{T_a} 1_{\{X_s>a-x\}}dB_s + \int_0^{T_a} 1_{\{X_s>a-x\}}\p(X_s,L_s^{X_s})ds +\frac{1}{2} Z^{(a)}_x.$$
Now one has $(X_{T_a}-(a-x))^+=x$, $M_x=\int_0^{T_a} 1_{\{X_s>a-x\}}dB_s$ is a martingale (in the spatial variable $x$) and the extended occupation formula gives
$$\int_0^{T_a} 1_{\{X_s>a-x\}}\p(X_s,L_s^{X_s})ds=\int_0^x h(a-z,Z^{(a)}_z)dz.$$
Thus $Z^{(a)}_x$ is a semimartingale, whose quadratic variation is $4 \int_0^x Z^{(a)}_z dz$ (see Theorem (2.2) p.455 in [RY]). This proves the proposition.
\end{proof}

\subsection{Criterion for recurrence}
In the rest of this section, we assume that $\p \in \Lambda_c$. We will note $h(x,l)\equiv h(l)$ for all $x$ and $l$.
Note that Theorem \ref{rayknight0} (Ray--Knight) implies that for $a>0$, we have
\bprop[Ray--Knight bis]
\label{rayknight}
The process $(Z^{(a)}_x,x\in[0,a])$ is a diffusion started at $0$ with generator $\kL=2zd^2/dz^2+2(1-h(z))d/dz$. Thus $Z^{(a)}$ solves the following SDE:
\begin{eqnarray}
\label{EDSRK}
dZ^{(a)}_x = 2 \sqrt{Z^{(a)}_x}\ d\beta_x + 2(1-h(Z^{(a)}_x))\ dx,
\end{eqnarray}
with $\beta$ a Brownian motion.
\eprop

In the following, $Z$ will denote a diffusion started at $0$ with generator $\kL$. Then
$(Z^{(a)}_x,0\le x\le a)$ and $(Z_x,0\le x\le a)$ have the same law.
As noticed in \cite{NRW} Theorem 3, the diffusion $Z$ admits an invariant measure $\pi$ given by
\begin{equation}\label{cstec}\pi(x)=c\cdot \exp \left[-\int_0^x h(l)\ \frac{dl}{l}\right],\end{equation}
with $c$ some positive constant. When $\pi$ is a finite measure, $c$ is chosen such that $\pi$ is a probability measure.
This implies the following generalization of Theorem \ref{theo}~:

\bthm
\label{theorec}
Let $\p \in \Lambda_c$. Then
$$\P^{\p,+}(R)=1 \quad \Longleftrightarrow \quad \int_0^\infty \exp \left[-\int_0^x h(l)\ \frac{dl}{l}\right] \ dx =+\infty.$$
\ethm

\begin{proof}
Assume that $\P^{\p,+}(R)=1$. Then $Z^{(a)}_a=L^{(1),0}_{T_a}$ converges a.s. towards $\infty$ as $a\to\infty$. Thus $Z_a$ (which is equal in law to $Z^{(a)}_a$) converges in probability towards $\infty$. Thus $Z$ is not positive recurrent, i.e. $\pi$ is not a finite measure.

Assume now that $\P^{\p,+}(R)=0$. Then $Z^{(a)}_a=L^{(1),0}_{T_a}$ converges a.s. towards $L^{(1),0}_\infty<\infty$ as $a\to\infty$. Thus $Z_a$  converges in law towards $L^{(1),0}_\infty$. The law of $L^{(1),0}_\infty$ is then an invariant probability measure. Since $Z$ is irreducible it admits at most one invariant probability measure. Thus $\pi$ is a probability measure. 
\end{proof}

Of course an analogue result holds as well for $X^2$. The criterion becomes:
$$\P^{\p,-}(R)=1 \quad \Longleftrightarrow \quad \int_0^\infty \exp \left[\int_0^x h(l)\ \frac{dl}{l}\right] \ dx =+\infty.$$
In particular, $X^1$ and $X^2$ cannot be both transient.

Observe now that if $\lim_{l\to +\infty} h(l)$ exists, and equals let say $h_\infty$, then
$$\P^\p_0(R)=1 \quad \Longrightarrow \quad h_\infty \in [-1,1],$$
and
$$\P^\p_0(R)=0 \quad \Longrightarrow \quad h_\infty \notin (-1,1).$$
But in the critical case $|h_\infty|=1$ both recurrence and transience regimes may hold. For instance
if $h(l)-1 \sim \alpha/\ln l$ in $+\infty$, then $\P^\p_0(R)=1$ if $\alpha < 1$, whereas $\P^\p_0(R)=0$ if $\alpha>1$.

\medskip
Let us finish this subsection with this last remark:
Let $\zeta_x:=L^{(1),x}_\infty$, with $x\ge 0$. Assume $Z$ is positive recurrent (then $X^1$ is transient). Using the fact that $Z$ is reversible, it can be seen that $\zeta$ is a diffusion with generator $\kL$ and initial distribution $\pi$ (see Theorem 3 in \cite{NRW}).

\subsection{Law of large numbers}
Our next result is a strong law of large numbers with an explicit expression for the speed.
\bthm
\label{LLN}
Let $\p \in \Lambda_c$.
\begin{itemize}
\item[(i)] Assume $\pi$ is a probability measure and denote by $v$ the mean of $\pi$
(i.e. $v=\int_0^\infty x\pi(x)dx\in [0,\infty]$). Then
$\tilde{\P}^\p$--a.s.
$$\lim_{t\to\infty}\frac{X^1_t}{t} = \frac{1}{v}.$$
\item[(ii)] Assume $\pi$ is an infinite measure. Then, $\tilde{\P}^\p$--a.s.
$$\lim_{t\to\infty}\frac{X^1_t}{t} = 0.$$
\end{itemize}
\ethm
\begin{proof} We take $v=\infty$ when $\pi$ is an infinite measure.
We first prove that (in all cases) $\tilde{\P}^\p$--a.s.
\begin{eqnarray}
\label{limiteTK}
\frac{T_K}{K}\to v,
\end{eqnarray}
when $K\to +\infty$.
Fix $N\ge 1$. Since $\p \in \Lambda_c$,
\begin{eqnarray}
\label{minoration}
\frac{1}{K}T_K \ge \frac{1}{K} \sum_{i=1}^{[K/N]-1} (T_{N(i+1)}-T_{Ni}) \ge \frac{1}{K} \sum_{i=1}^{[K/N]-1} U_i,
\end{eqnarray}
where the $U_i$ are i.i.d. random variables distributed like $T_N$
($U_i$ is the time spent in $[Ni,N(i+1)]$ between times $T_{Ni}$ and
$T_{N(i+1)}$). Now (in the following $\E$ denotes the expectation
with respect to $\tilde{\P}^\p$)
$$\E[Z_a]=\E[Z^{(a)}_a]=\E[L^{(1),0}_{T_a}]$$
is an increasing function of $a$. By monotone convergence $\E[Z_a]$ converges towards $\E[Z_\infty]\equiv v$.
It is standard to see by using \eqref{EDSRK} that $\E[Z_a]<\infty$ for all $a$ ($\p$ being bounded, $2(1-h(z))$ is dominated by $c(1+z)$ for some positive $c$).
Thus for all $N\ge 1$,
$$\E[T_N]=\int_0^N \E[L^{(1),a}_{T_N}]\ da= \int_0^N \E[Z_a]\ da\le N\E[Z_N] <+\infty.$$
Therefore \eqref{minoration} shows that
\begin{eqnarray*}
\liminf_{K\to\infty} \frac{1}{K}T_K &\ge& \frac{1}{N}\E[T_N]\\
&\ge& \frac{1}{N}\int_0^N \E[Z_a]\ da.
\end{eqnarray*}
By taking the limit as $N\to\infty$, we get
\begin{eqnarray}
\label{liminfTK}
\liminf_{K\to\infty} \frac{1}{K}T_K \ge v.
\end{eqnarray}
So we can conclude when $v=\infty$. Assume now that $v<\infty$. Then $\zeta_a=L^{(1),a}_{\infty}$ is a diffusion with generator $\kL$ and initial
distribution $\pi$.
Now
$$T_K= \int_0^K L^{(1),a}_{T_K}\ da \le \int_0^K \zeta_a\ da,$$
so that
\begin{eqnarray*}
\limsup_{K\to \infty} \frac{1}{K} T_K \le v,
\end{eqnarray*}
by the well known ergodic theorem for positive recurrent diffusions (see, for instance, \cite{IK} \S 6.8).
Together with \eqref{liminfTK} this proves \eqref{limiteTK}.

\medskip
Note that for all $t\in (T_N,T_{N+1})$.
\begin{eqnarray}
\label{inegalite1}
\frac{X^1_t}{t} \le \frac{N+1}{T_N}.
\end{eqnarray}
The law of large numbers follows immediately when $v=\infty$.
Assume now that $v<+\infty$.
Since $\p$ is bounded, for all $N$, $(X^1_t, t\ge T_N)$ dominates a Brownian motion with drift $-\|\p\|_\infty$ started at $N$ and absorbed at $0$. So
for all $\epsilon>0$, there exist constants $c>0$ and $C>0$ such that
$$\P^{\p,+}\left[\inf_{t\in (T_N,T_{N}+cN)} X_t < (1-\epsilon)N\right]  \le C \exp (-cN),$$
for all $N\ge 1$. So $\tilde{\P}^\p$-a.s. we have
$$ \inf_{t\in (T_N,T_{N}+cN)} X^1_t\ge (1-\epsilon)N,$$
for all $N$ large enough. Moreover since a.s. $\lim_{N\to\infty} T_N/N = v$, we have $T_{N+1}-T_N=o(N)$. So
a.s. for $N$ large enough,
$$ \inf_{t\in (T_N,T_{N+1})} X^1_t\ge (1-\epsilon)N.$$
Then a.s. for $N$ large enough and $t\in (T_N,T_{N+1})$,
$$\frac{X_t^1}{t} \ge (1-\epsilon)\frac{N}{T_{N+1}}.$$
By \eqref{limiteTK}, a.s. $T_N \sim v N$. So this
together with \eqref{inegalite1} finishes the proof of the theorem.
\end{proof}

Here also an analogous result holds for $X^2$. One has
$\tilde{\P}^\p$-a.s.
$$\frac{X^2_t}{t}\to -\frac{1}{v'},$$
where
$$v'= \int_0^\infty x\exp \left[\int_0^x h(l)\ \frac{dl}{l}\right] \ dx.$$
Now if $\p \in \Lambda_c$, as we already observed, $X^1$ and $X^2$ cannot be both transient. But assume for instance that, let say, $X^1$ is transient and $X^2$ is recurrent.
Then for $t$ large enough $X^1_t=X_{t+U}$, where $U$ is the finite random variable equal to the total time spent by $X$ in the negative part. So $X$ satisfies the same law of large numbers as $X^1$, namely a.s.
$$\frac{X_t}{t} \to \frac{1}{v}.$$

\brem \emph{Let us comment now on the case when $\p$ is only
supposed to be locally bounded. Assume that $\p$ is constant in the
first variable and that $\liminf_{l\to\infty} h(l)/l > -\infty$ (this includes the case $\inf_{l\ge 0}\p(0,l)>-\infty$). Then, as we
will see in the appendix below, the process $X^1$ of law $\P^{\p,+}$
is still well defined. Moreover, as the reader may check,
Proposition \ref{01law} and Theorem \ref{theorec}
hold as well for $X^1$. Concerning Theorem \ref{LLN}, it is still true that $T_K/K\to v$. Then the law of large numbers follows if $v=\infty$ or under the additional hypothesis that $\p$ is bounded below. Using the same method, like in \cite{NRW},
it is also possible to construct $\P^\p$, assuming that  $\liminf_{l\to\infty} h(l) > 0$.
Then the study of $\P^\p$ can be reduced to the study of
$\P^{\p,+}$. But this method does not look to apply to construct
$\P^{\p,-}$. We observe also that the following basic question remains:} \erem

\noindent \textbf{Question:} Is it possible to find $\p$ for which $\P^{\p,+}$ can be defined and such that $\P^{\p,+}(\sup_{t\ge 0} X_t<+\infty)>0$?

\section{Central limit theorem}\label{tlc}
In this section, $\p\in\Lambda_c$ and we assume that $\pi$ is a probability measure with finite mean $1/v$.
Define $s$ the scale function of the diffusion with generator $\kL$, by $s(1)=0$ and 
$$s'(x)=x^{-1}\exp\left(\int_0^x \frac{h(l)}{l}\ dl\right).$$
Note that $\lim_{x\to 0} s(x) = -\infty$ and $\lim_{x\to\infty} s(x)=+\infty$.
Let 
$$\sigma := 4\sqrt{c} v^{3/2} \left(\left(\int_{[0,1]^2}+\int_{[1,\infty)^2}\right) \frac{|s(x)|\wedge |s(y)|}{s'(x)s'(y)}\ dx \ dy\right)^{1/2},$$
with $c$ defined in \eqref{cstec}.

The following theorem is similar to Theorem D.5 in \cite{HY}:
\begin{theorem}
\label{theotcl}
Assume that $\sigma<+\infty$. Let $X$ be a process with law $\P^{\p,+}$. Then, as $t\to \infty$, 
$$\frac{X_t-vt}{\sigma\sqrt{t}}$$
converges in law toward a centered and reduced Gaussian variable. This convergence also holds for $\inf_{s\geq t}X_s$ and for $\sup_{s\le t}X_s$ instead of $X_t$.
\end{theorem}
\begin{proof}
Denote by $T_r$ the hitting time of $r$ by $X$.
Then, for all $0<K<r$, $$\P^{\p,+}\left[\inf_{s>T_r}X_s>r-K\right]=\P^{\p,+}\left[\inf_{s>T_K}X_s>0\right],$$ 
and this quantity converges to $1$ when $K\to\infty$. Thus it suffices to show the central limit theorem for $\sup_{s\le t} X_s$ in lieu of $X_t$. 
Then, up to notation (with in particular $-X$ in place of $X$), we can follow the proof of Theorem D.5 in \cite{HY}. 
Indeed we have $T_r=\int_0^rZ_t\ dt$, and it suffices to show that 
\begin{eqnarray}
\label{tclTr}
\frac{T_r-r/v}{\sqrt{r}}\longrightarrow \sigma v^{-3/2}\cN,
\end{eqnarray}
with $\cN$ a centered and reduced Gaussian variable.

Now there exists some Brownian motion $\beta$ such that $Z_t=s^{-1}(\beta(a_t^{-1}))$, where $a_t=\int_0^t\Delta(\beta_v)\ dv$ and $\Delta(x)=1/\left(4s^{-1}(x)s'^2(s^{-1}(x))\right)$. Now $\Delta$ is integrable on $\RR$ since $\pi$ is a finite measure:
$$\int_{-\infty}^\infty \Delta(x)\ dx=\int_0^\infty \frac{du}{4us'(u)}=(4c)^{-1}.$$
With notation of \cite{HY} this means that $C_9=4c>0$. In the same way, $v>0$ implies that $C_6$ from \cite{HY}, which is equal to $v$, is positive and $C_{10}=\sigma v^{-3/2}$ is finite if $\sigma<+\infty$. Then (see \cite{HY} for details) we get \eqref{tclTr}.
The theorem follows. 
\end{proof}

As observed for the law of large numbers, Theorem \ref{theotcl}
implies also a central limit theorem for an excited Brownian
motion of law $\P^\p$, when the process of law $\P^{\p,+}$ (or
$\P^{\p,-}$) has nonzero speed and satisfies the condition of Theorem \ref{theotcl}.

\begin{proposition}
\label{proptcl}
If $h(l)$ converges towards $h_\infty$, then the constant $\sigma$ is finite if  $h_\infty>4$ and infinite if $0<h_\infty<4$.
\end{proposition}
\begin{proof}
First
$$ I_1=\int_{[1,\infty)^2} \frac{|s(x)|\wedge |s(y)|}{s'(x)s'(y)}\ dx\ dy$$ is equal to
$$\int_1^\infty \left(\int_1^y \frac{s(x)}{s'(x)}\ dx + \int_y^\infty \frac{s(y)}{s'(x)} dx\right)\frac{dy}{s'(y)}.$$
But note that for any $\epsilon>0$, $x^{h_\infty-1-\epsilon}\le s'(x)\le x^{h_\infty-1+\epsilon}$ and $x^{h_\infty-\epsilon}\le s(x)\le x^{h_\infty+\epsilon}$ for $x$ large enough.
Thus for any $\epsilon>0$,
$$\frac{1}{s'(y)} \int_1^y \frac{s(x)}{s'(x)}\ dx,$$
and 
$$\frac{1}{s'(y)} \int_y^\infty \frac{s(y)}{s'(x)}\ dx,$$
both lie in the interval $[y^{3-h_\infty-\epsilon},y^{3-h_\infty+\epsilon}]$ for $y$ large enough. This implies that $I_1<\infty$ if $h_\infty>4$ and $I_1=\infty$ if $0<h_\infty<4$.
The term 
$$I_2 = \int_{[0,1)^2} \frac{|s(x)|\wedge |s(y)|}{s'(x)s'(y)}\ dx \ dy$$
is always finite, since as $x\to 0$, $s'(x)\sim 1/x$. \end{proof}

\section{Appendix: Construction of $\P^\p$ and $\P^{\p,+}$ in cases $\p$ is locally bounded}
We assume that $\p$ is constant in the first variable. We note $\p(x,l) \equiv \p(l)$ for any $x$ and $l$. We assume also that $\p$ is locally bounded, i.e. $\sup_{l\in [0,L]} \p(l)<+\infty$ for any $L>0$, and $\liminf_{l\to\infty} h(l)/l > -\infty$.
We want to define $\P^{\p,+}$. Following \cite{NRW}, set $\p_n(l):=1_{\{l\le n\}}\p(l)$.
Let $S'_n=\inf\{t\mid L^{X_t}_t>n\}$ and let $S_n=S'_n\wedge n$. According to Theorem $B$ in \cite{NRW}, if we can prove that for each $t>0$,
$$\P^{\p_n,+}(S_n<t)\to 0 \quad \textrm{as } n\to +\infty,$$
then by consistency, it will be possible to define $\P^{\p,+}$.
Observe that it is enough to prove the two statements:
$$\textrm{for each }a>0 \quad P^{\p_n,+}(S_n<T_a)\to 0 \quad \textrm{as } n\to +\infty,$$
and for each $t>0$ and $\epsilon>0$, there is some $a>0$ such that
$$\P^{\p_n,+}(T_a<t)<\epsilon \quad \textrm{for all }n.$$
The first statement can be proved like in \cite{NRW}: by using the Ray-Knight theorem under $\P^{\p_n,+}$,
$$\P^{\p_n,+}(S_n<T_a)=\P^{\p_n,+}\left(\sup_{0\le x\le a} Z_x^{(a)}>n\right)=\P^Z\left(\sup_{0\le x\le a}Z_x>n\right),$$
where $\P^Z$ denotes the law of the diffusion with generator $\kL=2zd^2/dz^2+2(1-h(z))d/dz$.
The last term tends to $0$ when $n\to +\infty$, since a.s. $Z$ does not explode in finite time (this follows from the fact that $\liminf_{l\to\infty} h(l)/l > -\infty$).
For the second statement, write:
\begin{eqnarray*}
\P^{\p_n,+}(T_a<t)&=&\P^{\p_n,+}\left(\int_0^a Z_x^{(a)}\ dx<t\right)\\
               &\le & \P^{\p_n,+}\left(\int_0^{a\wedge T_n} Z_x^{(a)}\ dx<t\right)\\
                &\le & \P^Z\left(\int_0^{a\wedge T_n} Z_x\ dx<t\right),
\end{eqnarray*}
where $T_n$ denotes the first time $Z^{(a)}$ or $Z$ reaches $n$. But
by using standard scale functions arguments, we can see $0$ is an
entrance boundary point for the diffusion with generator $\kL$. Thus
$\P^Z$-a.s.
$$\int_0^\infty Z_x\ dx=+\infty.$$
The second statement follows.

\medskip
We give now some hints
to construct $\P^\p$ when $\liminf h(l)>0$.
Using the fact that $\P^{\p_n,-}(R)=1$, all the stopping times $T_a$ are finite $\P^{\p_n}$-a.s. for all positive $a$. Then the method of \cite{NRW}
can be followed.

\end{document}